\documentclass[11pt, reqno]{article}
\usepackage{amsmath,amssymb,amsfonts,amsthm}
\usepackage{color}

\usepackage[colorlinks=true,linkcolor=blue,citecolor=blue,urlcolor=blue,pdfborder={0 0 0}]{hyperref}
\usepackage{cleveref}

\usepackage{caption}
\usepackage{thmtools}

\usepackage{mathrsfs}

\crefname{theorem}{Theorem}{Theorems}
\crefname{thm}{Theorem}{Theorems}
\crefname{lemma}{Lemma}{Lemmas}
\crefname{lem}{Lemma}{Lemmas}
\crefname{remark}{Remark}{Remarks}
\crefname{prop}{Proposition}{Propositions}
\crefname{defn}{Definition}{Definitions}
\crefname{corollary}{Corollary}{Corollaries}
\crefname{conjecture}{Conjecture}{Conjectures}
\crefname{question}{Question}{Questions}
\crefname{chapter}{Chapter}{Chapters}
\crefname{section}{Section}{Sections}
\crefname{figure}{Figure}{Figures}

\theoremstyle{plain}
\newtheorem{thm}{Theorem}[section]

\newtheorem{lemma}[thm]{Lemma}

\newtheorem{corollary}[thm]{Corollary}

\newtheorem{prop}[thm]{Proposition}

\newtheorem{question}[thm]{Question}
\theoremstyle{definition}

\theoremstyle{remark}
\newtheorem{remark}[thm]{Remark}

\numberwithin{equation}{section}

\renewcommand{\P}{\mathbb P}

\newcommand{\R}{\mathbb R}
\newcommand{\Z}{\mathbb Z}

\newcommand{\cH}{\mathcal H}

\newcommand{\bbH}{\mathbb H}

\usepackage[margin=1.05in]{geometry}
\usepackage{extarrows}
\usepackage{commath}
\usepackage{mathtools}
\newcommand{\eps}{\varepsilon}
\usepackage{bbm}
\usepackage{setspace}
\usepackage{enumitem}
\usepackage{tikz}

\newcommand{\Aut}{\operatorname{Aut}}

\renewcommand{\root}{0}

\newcommand{\stab}{\operatorname{Stab}}
\newcommand{\sgn}{\operatorname{sgn}}

\title{Self-avoiding walk on nonunimodular transitive graphs}

\author{Tom Hutchcroft\footnote{Statistical Laboratory, DPMMS, Univeristy of Cambridge}}

\begin{document}

\maketitle

\vspace{-1em}

\begin{abstract}
We study self-avoiding walk on graphs whose automorphism group has a transitive nonunimodular subgroup. We prove that 
 	self-avoiding walk is ballistic, that
 	the bubble diagram converges at criticality, and that
 	the critical two-point function decays exponentially in the distance from the origin.
This implies that the critical exponent governing the susceptibility takes its mean-field value, and hence that the number of self-avoiding walks of length $n$ is comparable to the $n$th power of the connective constant.
We also prove that the same results hold for a large class of repulsive walk models with a self-intersection based interaction, including the weakly self-avoiding walk. 
All these results apply in particular to the product $T_k \times \Z^d$ of a $k$-regular tree ($k\geq 3$) with $\Z^d$, 
 for which these results were previously only known for large $k$.
\end{abstract}

\section{Introduction}

A \textbf{self-avoiding walk} (SAW) on a graph $G$ is a path in $G$ that visits each vertex at most once. 
In the probabilistic study of self-avoiding walk, one fixes a graph (often the hypercubic lattice $\Z^d$), and is interested in both  enumerating the number of $n$-step SAWs and studying the asymptotic behaviour of a uniformly random SAW of length $n$. 
This leads to two particularly important questions.
  \begin{question}
\label{Q:counting}
  What is the asymptotic rate of growth of the number of SAWs of length $n$?
\end{question}
\begin{question}
\label{Q:speed}
How far from the origin is the endpoint of a typical SAW of length $n$?
\end{question}
These questions are simple to state but are often very difficult to answer.
Substantial progress has been and continues to be made for SAW on Euclidean lattices. 
In particular, a very thorough understanding of SAW on $\Z^d$ for $d\geq 5$ has been established in the seminal work of Hara and Slade 
\cite{MR1171762,MR1174248}.  
The low-dimensional cases $d=2,3,4$ continue to present serious challenges.
 For a comprehensive introduction to and overview of this literature, we refer the reader to \cite{MR2986656,MR3025395}.

%
%

Recently, the study of SAW on more general graphs has gathered momentum. In particular, a systematic study of SAW on transitive graphs has been initiated in a series of papers by Grimmett and Li \cite{MR3118955,1610.00107,1412.0150,1510.08659,MR3646785,MR3252803,MR3367126}, which is primarily concerned with properties of the connective constant. Other works on SAW on non-Euclidean transitive graphs include \cite{MR3166203,MR3005730,benjamini2016self,li2016positive,madras2005self,gilch2017counting}. See \cite{1704.05884} for a survey of these results.


In this paper, we given complete answers to \cref{Q:counting} and \cref{Q:speed} for self-avoiding walk on graphs whose automorphism group admits a \emph{nonunimodular transitive subgroup} (defined in the next subsection).  Although graphs whose \emph{entire} automorphism group is nonunimodular are generally considered to be rather contrived and unnatural, the class of graphs with a nonunimodular transitive \emph{subgroup} of automorphisms is much larger. Indeed, it includes natural examples such as the product $T_k \times \Z^d$ of a $k$-regular tree with $\Z^d$ for every $k\geq 3$ (or indeed $T_k \times H$ where $H$ is an arbitrary transitive graph), for which the results of this paper were only previously known for sufficiently large values of $k$ (see the discussion at the end of this subsection). 

Our proofs are inspired by the analysis we carried out for percolation on the same class of graphs in our paper \cite{Hutchcroftnonunimodularperc}, which relies on similar methodology. It should be remarked that although every graph possessing a transitive nonunimodular subgroup of automorphisms is necessarily nonamenable \cite{MR1082868}, we never use this fact in our analysis.

 Our first theorem answers \cref{Q:counting} in the nonunimodular  context. 
Let $G$ be a transitive graph, let $\root$ be a fixed root vertex of $G$, and let $Z(n)$ be the number of length $n$ SAWs in $G$ starting at $\root$. 
Hammersley and Morton \cite{MR0064475} observed that $Z(n)$ satisfies the submultiplicative inequality $Z(n+m)\leq Z(n)Z(m)$, from which it follows by Fekete's Lemma that there exists a constant $\mu_c=\mu_c(G)$, known as the \textbf{connective constant} of $G$, such that 
\[
\mu_c = \lim_{n\to\infty} Z(n)^{1/n} = \inf_{n\geq 1} Z(n)^{1/n},
\]
 so that in particular
\[\mu_c^n \leq Z(n) \leq \mu_c^{n+o(n)}\]
for every $n\geq 0$. We also define the
\textbf{susceptibility} $\chi(z)$ to be the generating function
\[
\chi(z) = \sum_{n\geq 0} z^n Z(n),
\]
which has radius of convergence $z_c:=\mu_c^{-1}$.   The connective constant is not typically expected to have a nice or interesting value (a notable exception is the hexagonal lattice \cite{MR675241,MR2912714}), and it is usually much more interesting to estimate the subexponential correction to $Z(n)$ than it is to estimate $\mu_c$.  We stress that submultiplicativity arguments alone do not yield any control of this subexponential correction whatsoever.


\begin{thm}[Counting walks]
\label{thm:SAWcounting}
Let $G$ be a connected, locally finite graph, and suppose that $\Aut(G)$ has a transitive nonunimodular subgroup. Then there exists a constant $C$ such that
\[\frac{z_c}{z_c-z} \leq \chi(z) \leq \frac{C z_c}{z_c-z} \quad \text{ and } \quad
\mu_c^n \leq Z(n) \leq C\mu_c^n
\] 
for every $0\leq z<z_c$ and $n\geq 0$.
\end{thm}

In fact we are able to obtain explicit estimates on the constants that appear in this theorem, see \cref{remark:quant}. 
The lower bounds of the theorem are trivial consequences of submultiplicativity.
The upper bounds on $\chi(z)$ and $Z(n)$ in this theorem are equivalent up to the choice of constant, see \cref{lem:chitoZ}. Probabilistically, the upper bound on $Z(n)$ means that the concatenation of two uniformly chosen $n$-step SAWs has probability at least $1/C>0$ to be self-avoiding for every $n\geq 0$.

Our next theorem answers \cref{Q:speed} in the nonunimodular context. 
We define $\P_n$ to be the uniform measure on self-avoiding walks of length $n$ in $G$ starting at $\root$, and denote the random self-avoiding walk sampled from $\P_n$ by $X=(X_i)_{i=0}^n$. For each $z\geq 0$ and $x \in V$, we define the \textbf{two-point function}
\[
G(z;x) = \sum_{\omega\in \Omega} z^{|\omega|} \mathbbm{1}\bigl( \omega: \root \to x \text{ self-avoiding}\bigr) = \sum_{n\geq 0} z^{n} Z(n) \P_n\left( X_n = x\right).
\]
In the following theorem, $d(\root,x)$ denotes the graph distance between $\root$ and $x$.

\begin{thm}[Speed and two-point function decay]
\label{thm:SAWspeed}
Let $G=(V,E)$ be a connected, locally finite graph, and suppose that $\Aut(G)$ has a transitive nonunimodular subgroup. Then there exists a positive constant $c$ such that
  \[
  G(z_c;x) \leq \exp\left[-c d(0,x)\right]
  \]
for every $x\in V$ and 
  \[
\P_{n}\left(d(0,X_n) \geq c n\right) \geq 1-\exp\left[-c n\right]
  \]
  for every $n\geq 0$.
\end{thm}

Let us briefly survey related theorems in the literature. It is reasonable to conjecture that the conclusions of \cref{thm:SAWcounting,thm:SAWspeed} hold for every transitive nonamenable graph. Indeed, it is plausible that the conclusion of \cref{thm:SAWcounting} holds for every transitive graph with at least quintic volume growth. The conjectures are trivial when the graph is a tree. Li \cite{li2016positive} has shown that SAW is ballistic on a certain class of infinitely ended transitive graphs, and Madras and Wu \cite{madras2005self} and Benjamini \cite{benjamini2016self} have shown that SAW on certain specific hyperbolic lattices has linear mean displacement; the latter paper also establishes that the conclusions of \cref{thm:SAWcounting} hold for the lattices they consider. Gilch and M\"uller \cite{gilch2017counting} have proven that the conclusions of \cref{thm:SAWcounting} hold for free products of quasi-transitive graphs (which are always infinitely ended).

 Nachmias and Peres \cite{MR3005730} proved that  the conclusions of \cref{thm:SAWcounting,thm:SAWspeed} hold for every transitive nonamenable graph satisfying 
\begin{equation}
\label{eq:NachPe} (d-1)\rho  < \mu_c, \end{equation}
where $d$ is the degree of the graph, $\rho$ is its spectral radius, and $\mu_c$ is its connective constant. In particular, this holds whenever $\rho \leq 1/2$ \cite{MR1756965}, as well as for nonamenable transitive graphs of large girth (where what constitutes `large' depends on the spectral radius and the degree). However, we do \emph{not} expect this criterion to apply to $T_k \times \Z^d$ for small $k$. Furthermore, our non-perturbative approach also lets us handle \emph{anisotropic} SAW on $T_k \times \Z^d$, in which the walk is weighted to prefer $\Z^d$-edges to $T_k$-edges, and this bias may be arbitrarily strong (see \cref{subsec:general}). It seems very unlikely that such a result could be established  using perturbative techniques. 

For the hypercubic lattice, Hara and Slade \cite{MR1171762,MR1174248} proved that $Z(n)$ grows like $\mu_c^n$ as $n\to\infty$ whenever $d \geq 5$. In the same setting, they also proved that the distance from the origin to the endpoint of an $n$-step SAW is typically of order $n^{1/2}$. Hara \cite{MR2393990} later proved that the critical two-point function decays like $\|x\|^{-d+2}$. (Both behaviours are the same as for simple random walk.) For $d=4$ it is conjectured that similar asymptotics hold up to logarithmic corrections. See \cite{MR3339164} and references therein for an account of substantial recent progress on four dimensional \emph{weakly} self-avoiding walk. For $d=2,3$ the gap between what is known and what is conjectured is very large; important results include those of \cite{MR0139535,MR0166845,MR3117515,MR3474464}. See \cite{MR3025395,MR2986656} and references therein for more details.

\subsection{Tilted walks and the modular function}

We now define unimodularity and nonunimodularity. 
Let $G=(V,E)$ be a connected, locally finite graph, and let $\Aut(G)$ be the group of automorphisms of $G$.  Recall that a subgroup $\Gamma \subseteq \Aut(G)$ is said to be \textbf{transitive} if it acts transitively on $G$, that is, if for any two vertices $u,v\in V$ there exists $\gamma \in \Gamma$ such that $\gamma u = v$.
  The \textbf{modular function} $\Delta=\Delta_\Gamma: V^2 \to (0,\infty)$ of a transitive subgroup $\Gamma\subseteq \Aut(G)$ is defined to be
\[
\Delta(x,y)=\frac{|\stab_y x|}{|\stab_x y|},
\]
where $\stab_x y$ is the orbit of $y$ under the stabilizer of $x$ in $\Gamma$. The group $\Gamma$ is said to be \textbf{unimodular} if $\Delta \equiv 1$, and \textbf{nonunimodular} otherwise. The most important properties of the modular function are the \textbf{cocycle identity}, which states that
\[
\Delta(x,y)\Delta(y,z) = \Delta(x,z)
\]
for every $x,y,z\in V$, 
and the \textbf{tilted mass-transport principle}, which states that if $F:V^2\to[0,\infty]$ is invariant under the diagonal action of $\Gamma$, meaning that $F(\gamma x, \gamma y)=F(x,y)$ for every $\gamma \in \Gamma$ and $x,y \in V$, then
\[
\sum_{y \in V} F(x,y) = \sum_{y\in V}F(y,x)\Delta(x,y).
\]
See \cite[Chapter 8]{LP:book} for proofs of these properties and further background, and \cite[\S 4]{BC2011} for a probabilistic interpretation of the modular function. Note that $\Delta$ is itself invariant under the diagonal action of $\Gamma$. 

The prototypical example of a pair $(G,\Gamma)$ of a graph together with a nonunimodular transitive subgroup $\Gamma \subseteq \Aut(G)$ is given by the $k$-regular tree $T_k$ with $k\geq 3$ together with the group $\Gamma_\xi$ of automorphisms fixing some specified end $\xi$ of $T_k$. (An \textbf{end} of a tree is an equivalence class of infinite simple paths, where paths may start from any vertex and two simple paths are considered equivalent if the sets of vertices they visit have finite symmetric difference.) Let us briefly give an explicit description of the modular function in this example. Every vertex $v$ of $T_k$ has exactly one neighbour that is closer to the end $\xi$ than it is. We call this vertex the \emph{parent} of $v$. (In other words, the parent of $v$ is the unique neighbour of $v$ that lies in the unique simple path that starts at $v$ and is in the equivalence class $\xi$.) All other neighbours of $v$ are said to be \emph{children} of $v$. This leads to a partition of $T$ into levels $(L_n)_{n\in \Z}$, unique up to choice of index, such that if $v$ is in $L_n$ then its parents are in $L_{n+1}$ and its children are in $L_{n-1}$.  The modular function in this example is given explicitly by \[\Delta(u,v)=(k-1)^{n} \iff u \in L_m \text{ and } v \in L_{m+n} \text{ for some $m\in \Z$.}\]  
From this example many further examples can be built. In particular, if $G$ is an arbitrary transitive graph and $T_k \times G$ is the product of $T_k$ with $G$, then $\Aut(T_k \times G)$ has a nonunimodular transitive subgroup of automorphisms isomorphic $\Gamma$ to $\Gamma_\xi \times \Aut(G)$ and with modular function $\Delta_\Gamma ((u,x),(v,y))=\Delta_{\Gamma_\xi}(u,v)\Delta_{\Aut(G)}(x,y)$. 
See e.g.\ \cite{MR1856226,timar2006percolation,Hutchcroftnonunimodularperc} for further examples. 

As in \cite{Hutchcroftnonunimodularperc}, the key to our analysis is to define \emph{tilted} versions of classical quantities such as the susceptibility. These quantities will be similar to their classical analogues, but will have an additional parameter, $\lambda$, and will be weighted in some sense by the modular function to the power $\lambda$. We will show that these tilted quantities behave in similar ways to their classical analogues (corresponding to $\lambda=0$) but, crucially, will have different critical values associated to them. 

For each $\lambda \in \R$ and $n\geq 0$, we define
\begin{align*}
Z(\lambda; n) &= \sum_{x\in V} \sum_{\omega \in \Omega} \mathbbm{1}\bigl[ \omega : \root \to x \text{ a length $n$ SAW}\bigr]\Delta^\lambda(\root,x),\\
\intertext{and define the \textbf{tilted susceptibility} to be}
\chi_w(z,\lambda) &= \sum_{n\geq0} z^n Z_w(\lambda; n).
\end{align*}
Since every self-avoiding walk $\omega$ of length $n+m$ is the concatenation of two self-avoiding walks $\omega_1$ and $\omega_2$ of lengths $n$ and $m$ respectively, the cocycle identity implies that
\[
Z(\lambda;n+m) \leq Z(\lambda;n)Z(\lambda;m)
\]
for every $n,m\geq 0$. It follows by Fekete's Lemma, as before, that 
\begin{equation}
\label{eq:muclambdadef}
\mu_{c,\lambda} = \mu_{c,\lambda}(G,\Gamma):=\lim_{n\to\infty} Z(\lambda;n)^{1/n} = \inf_{n\geq 1} Z(\lambda;n)^{1/n}
\end{equation}
and that $z_{c,\lambda}=z_{c,\lambda}(G,\Gamma):=\mu_{c,\lambda}^{-1}$ is the radius of convergence of $\chi(z;\lambda)$. 

The tilted mass-transport principle leads to a symmetry between $\lambda$ and $1-\lambda$. Indeed, it implies that
\begin{align}
\nonumber
Z(\lambda;n)&= \sum_{x\in V} \sum_{\omega \in \Omega} \mathbbm{1}\bigl[ \omega : \root \to x \text{ a length $n$ SAW}\bigr]\Delta^\lambda(\root,x) \\
&=\sum_{x\in V}\sum_{\omega \in \Omega} \mathbbm{1}\bigl[ \omega : x\to \root \text{ a length $n$ SAW}\bigr]\Delta^\lambda(x,\root) \Delta(\root,x) \nonumber \\
&=\sum_{x\in V}\sum_{\omega \in \Omega} \mathbbm{1}\bigl[ \omega : \root \to x \text{ a length $n$ SAW}\bigr]\Delta^{1-\lambda}(\root,x) =Z(1-\lambda;n) \label{eq:lambdasymmetry}
\end{align}
for every $\lambda \in \R$ and $n\geq 0$, and hence that
\[
 \chi(z,\lambda)=\chi(z,1-\lambda)
\]
for every $\lambda \in \R$, and $z\geq 0$. In particular, it follows that $z_{c,\lambda}=z_{c,1-\lambda}$ for every $\lambda \in \R$.  
Moreover, it is easy to see that $Z(\lambda,n)$ is a convex function of $\lambda$ for each fixed $n$, and combined with \eqref{eq:lambdasymmetry} this implies that both $Z(\lambda,n)$ and $\chi_w(z,\lambda)$ are decreasing on $(-\infty,1/2]$ and increasing on $[1/2,\infty)$, while $z_{c,\lambda}(w)$ is increasing on $(-\infty,1/2]$ and decreasing on $[1/2,\infty)$. This leads to a special role for $\lambda=1/2$, which we call the \textbf{critical tilt}. We call $z_t=z_{c,1/2}$ the \textbf{tiltability threshold} and call $[0,z_t)$ the \textbf{tiltable phase}.

The main technical result of this paper is the following. We will show in \cref{sec:profit} that it easily implies \cref{thm:SAWcounting,thm:SAWspeed}.
\begin{thm}
\label{thm:lambda}
Let $G$ be a connected, locally finite graph, and let $\Gamma$ be a transitive nonunimodular subgroup of $\Aut(G)$.
 Then the function 
 \[ \R \to (0,z_t],\quad \lambda \mapsto z_{c,\lambda}\] is continuous, and is strictly increasing on $(-\infty,1/2]$. 
\end{thm}

Besides implying \cref{thm:SAWcounting,thm:SAWspeed}, \cref{thm:lambda} also immediately yields \emph{tilted} versions of those theorems, at least when $\lambda$ does not take its critical value of $1/2$. 
We define a probability measure on the set of self-avoiding walks of length $n$ starting at $\root$ by
\begin{align*}
\P_{\lambda,n}(\{\omega\}) &= Z(\lambda;n)^{-1} \Delta^\lambda(\root,\omega^+),
\end{align*}
where $\omega^+$ denotes the endpoint of $\omega$.

\begin{thm}
\label{thm:tilted}
Let $G$ be a connected, locally finite graph, let $\Gamma$ be a transitive nonunimodular subgroup of $\Aut(G)$. Then the following hold.
\begin{enumerate}
  \item For every $\lambda \neq 1/2$, there exists a constant $C_\lambda$ such that
\[\frac{z_{c,\lambda}}{z_{c,\lambda}-z} \leq \chi(z,\lambda) \leq \frac{C_\lambda z_{c,\lambda}}{z_{c,\lambda}-z} \quad \text{ and } \quad
\mu_c^n \leq Z_w(\lambda;n) \leq C_\lambda \mu_{c,\lambda}^n
\] 
for every $0<z<z_{c,\lambda}$ and $n\geq 0$.
  \item For every $z \in [0,z_t)$, there exists a constant $c_{z}$ such that
  \[
  G(z;x) \leq e^{-c_{z} d(0,x)}
  \]
  for every $x\in V$.
  \item For every $\lambda \neq 1/2$, there exist positive constants $c_\lambda$ and $c_\lambda'$ such that
  \begin{equation}
  \label{eq:speedheight}
\P_{\lambda,n}\Bigl[\sgn(\lambda -1/2) \log \Delta(0,X_n) \geq c_\lambda n\Bigr] \geq 1-e^{-c_\lambda n}
  \end{equation}
  for every $n\geq 0$ and hence
  \begin{equation}
  \label{eq:speednotheight}
\P_{\lambda,n}\Bigl[d(0,X_n) \geq c_\lambda' n\Bigr] \geq 1-e^{-c_\lambda' n}.
  \end{equation}
  for every $n\geq 0$.
\end{enumerate}
\end{thm}

\begin{remark} The estimate \eqref{eq:speedheight} also yields additional information concerning the untilted case $\lambda =0$.
Indeed, it shows that for all $\lambda \neq 1/2$ the \emph{height} $\log \Delta(\root,X_n)$ of the walk behaves ballistically. Moreover, it shows that the walk is typically displaced in the downward direction when $\lambda<1/2$ and the upward direction when $\lambda>1/2$.
\end{remark}

\cref{thm:tilted} naturally leads to questions concerning the critically tilted case $\lambda=1/2$. We present some such questions along with some partial results in \cref{sec:end}.

\subsection{Other repulsive walk models}
\label{subsec:general}

All our results apply more generally to a large family of repulsive walk models, including the self-avoiding walk as a special case. This generalization does not add substantial complications to the proof. In this section we define the family of models that we will consider and state the generalized theorem.

Let $G$ be a connected, locally finite graph, let $\Gamma$ be a transitive subgroup of $\Aut(G)$, and let $E^\rightarrow$ be the set of oriented edges of $G$. 
An oriented edge $e$ of $G$ is oriented from its tail $e^-$ to its head $e^+$ and has reversal $e^\leftarrow$. 
Let $n\geq0$. A \textbf{path of length $n$} in $G$ is a pair of functions $\omega\{0,\ldots,n\} \to V$ and $\omega\{ (i,j) : 0\leq i,j \leq n : |i-j|=1\} \to E^\rightarrow$ such that $\omega(i+1,i) =\omega(i,i+1)^\leftarrow$ for every $0\leq i < n$ and 
$\omega(i,i+1)^-=\omega(i)$ and $\omega(i,i+1)^+=\omega(i+1)$ for every $0 \leq i < n$. (Note that, by a slight abuse of notation, we denote both the path and each of these functions by the same letter, usually $\omega$.)
In other words, a path of length $n$ is a multi-graph homomorphism from the line graph with $n$ edges into $G$. 
A \textbf{path} in $G$ is a path of some length $n$, and we write $|\omega|$ for the length of the path $\omega$. 
Note that length one paths are just oriented edges, while length zero (a.k.a.\ \textbf{trivial}) paths are just vertices.
We write $\omega^-$ and $\omega^+$ for the first and last vertices of $\omega$, and write $\omega^\leftarrow$ for the reversal of $\omega$. We say that an ordered pair of paths $(\omega_1,\omega_2)$ in $G$ are \textbf{contiguous} if $\omega_1^+=\omega_2^-$. We will usually write simply that $\omega_1$ and $\omega_2$ are contiguous, with the ordering being implicit. Given a pair of contiguous paths $(\omega_1,\omega_2)$, we define their concatenation $\omega_1\circ \omega_2$ in the natural way. By slight abuse of terminology, we say that a contiguous pair of paths $(\omega_1,\omega_2)$ are \textbf{disjoint} if $\omega_1(i) \neq \omega_2(j)$ for every $i < |\omega_1|$ and $j>0$.

Let $G$ be a graph, let $\Gamma$ be a transitive group of automorphisms of $G$, and let $\Omega$ be the set of paths of finite length in $G$.  Consider a \textbf{weight function} $w:\Omega \to [0,\infty)$, which is always taken to give weight one to every trivial path. 
We say that 
$w$ is $\Gamma$\textbf{-invariant} if $w(\gamma \omega)=w(\omega)$ for every $\gamma \in \Gamma$ and $\omega \in \Omega$. We say that $w$ is \textbf{reversible} if  $w(\omega) = w\left(\omega^\leftarrow\right)$ for every $\omega \in \Omega$. 
We say that $w$ is 
\textbf{repulsive} if
\[w(\omega \circ \eta) \leq w(\omega)w(\eta)\]
for every contiguous pair $(\omega,\eta) \in \Omega^2$. 
 We say that a weight function $w:\Omega \to [0,\infty)$ is \textbf{non-degenerate} if $w(\omega)=1$ for every $\omega$ with $|\omega|=0$, and $w(\omega)\geq 1$ for every $\omega$ with $|\omega|=1$.
We say that the weight function $w$ is \textbf{zero-range} if  $w(\omega \circ \eta) = w(\omega)w(\eta)$ whenever $\omega$ and $\eta$ are contiguous and disjoint. 

For brevity, we will call a weight function \textbf{good} if it is $\Gamma$-invariant, reversible, non-degenerate, zero-range, and   repulsive. 
Important examples of good weight functions include
\begin{align*}
w(\omega) &= \mathbbm{1}\left[ \omega \text{ is self-avoiding}\right]
\end{align*}
which is the weight function for self-avoiding walk, and
\begin{align*}
w(\omega) &= \exp\left[-g \sum_{0 \leq i < j \leq |\omega|} \mathbbm{1}\left(\omega(i)=\omega(j)\right)  \right]
\end{align*}for $g\geq0$, which is the weight function for weakly self-avoiding walk (a.k.a.\ the Domb-Joyce model~\cite{domb1972cluster}).  
Another very natural example is the \emph{anisotropic} self-avoiding walk on $T \times \Z^d$:
\begin{equation*}
w(\omega) = \mathbbm{1}\bigl[\omega \text{ self-avoiding}\bigr] \exp\left[ a\,\#\{\text{$\Z^d$ edges used by $\omega$}\} + b\,\#\{\text{tree edges used by $\omega$}\}\right] 
\end{equation*}for $a,b\geq 0$. This walk prefers to use $T$ edges if $a<b$ and prefers to use $\Z^d$ edges if $a > b$.
Finally, to illustrate the flexibility of the definition, let us also give examples of some more exotic good weight functions:
\begin{align*}
w(\omega) & = \mathbbm{1}\left[\text{$\omega$ does not visit any vertex more than twice} \right],\\
w(\omega) &= \mathbbm{1}\left[ \omega(i) \neq \omega(j) \text{ if } |i-j| \text{ is not prime}\right], \\
w(\omega) &= \mathbbm{1}\left[\text{the subgraph of $G$ spanned by the edges traversed by $\omega$ is a tree}\right].
\end{align*}
(Tyler Helmuth has informed us that the first of these examples is related to $\phi^6$ field theory.)

For each $\Gamma$-invariant weight function $w$, we define
\[
Z_w(n) = \sum_{\omega \in \Omega} w(\omega) \mathbbm{1}\bigl[\omega^-=0,\, |\omega|=n\bigr]
\]
to be the total weight of all walks of length $n$. For every $n$ with $Z_w(n)>0$, we define a probability measure on the set of paths of length $n$ starting at $\root$ by
\[
\P_{w,n}(\{\omega\}) = Z_w(n)^{-1}w(\omega).
\]
The two-point function $G_w(z;x)$, the susceptibility $\chi_w(z)$, and the critical parameter $z_c(w)$ are defined analogously to the case of self-avoiding walk, as are the tilted variants $\chi_w(z,\lambda),$  $z_{c,\lambda}(w)$, and $\P_{w,\lambda,n}$. 
Note that the tilted quantities all depend implicitly on the choice of $\Gamma$. Similarly to the case of SAW, if $w$ is $\Gamma$-invariant and reversible then the tilted mass-transport principle implies that
\[
Z_w(\lambda;n) = Z_w (1-\lambda; n),
\]
for every $\lambda \in \R$ and $n\geq 0$, 
and hence that  $\chi_w(\lambda;z)=\chi_w(1-\lambda;z)$ and $z_{c,\lambda}(w)=z_{c,1-\lambda}(w)$ for every $\lambda \in \R$. The same statements concerning  monotonicity of these quantities on $(-\infty,1/2]$ and $[1/2,\infty)$ hold for arbitrary good weight functions as they did for SAW, and for the same reasons.

 If $w$ is $\Gamma$-invariant and   repulsive then $Z_w(\lambda;n)$ is submultiplicative for every $\lambda$ by the cocycle identity, so that
\[
\mu_{c,\lambda}(w) := z_{c,\lambda}(w)^{-1} = \lim_{n\to\infty} Z_w(\lambda;n)^{1/n} = \inf_{n\geq 1} Z_w(\lambda;n)^{1/n} \in [-\infty,\infty)
\]
is well-defined by Fekete's Lemma. If furthermore $w$ is good then the non-degeneracy and zero-range properties imply the lower bound $\mu_{c,\lambda}\geq 1$ for every $\lambda \in \R$.

We will prove the following generalization of \cref{thm:lambda} to arbitrary good weight functions. Again, we stress that this applies in particular to any graph of the form $T \times G$, where $T$ is a regular tree of degree at least $3$ and $G$ is an arbitrary transitive graph.

\begin{thm}
\label{thm:generallambda}
Let $G$ be a connected, locally finite graph, let $\Gamma$ be a transitive nonunimodular subgroup of $\Aut(G)$, and let $w:\Omega\to [0,\infty)$ be a good weight function. 
 Then the function 
 \[ \R \to (0,z_t(w)],\quad \lambda \mapsto z_{c,\lambda}(w)\] is continuous, and is strictly increasing on $(-\infty,1/2]$. 
\end{thm}

\cref{thm:generallambda} has the following straightforward consequences, which generalize \cref{thm:tilted}. In particular, analogues of \cref{thm:SAWcounting,thm:SAWspeed} for general good weight functions $w$ follow from the untilted case $\lambda=0$.

\begin{thm}
\label{thm:generalrepulsive}
Let $G$ be a connected, locally finite graph, let $\Gamma$ be a transitive nonunimodular subgroup of $\Aut(G)$, and let $w:\Omega\to [0,\infty)$ be a good weight function. Then the following hold.
\begin{enumerate}
  \item For every $\lambda \neq 1/2$, there exists a positive constant $C_\lambda$ such that
  \[
  \frac{ z_{c,\lambda}}{z_{c,\lambda}-z} \leq  \chi_{w}(z,\lambda) \leq \frac{C_\lambda z_{c,\lambda}}{z_{c,\lambda}-z} 
\quad\text{ and }\quad
  \mu_c^n \leq Z_{w}(\lambda;n) \leq C_\lambda \mu_{c,\lambda}^n.\]
  for every $0<z<z_{c,\lambda}=z_{c,\lambda}(w)$ and $n\geq 0$.

  \item For every $z \in [0,z_t)$, there exists a positive constant $c_{z}$ such that
  \[
  G_w(z;x) \leq e^{-c_{z} d(0,x)}
  \]
  for every $x\in V$.
  \item For every $\lambda \neq 1/2$, there exist positive constants $c_\lambda,c_\lambda'$ such that
  \[
\P_{w,\lambda,n}\Bigl[\sgn(\lambda -1/2) \log \Delta(0,X_n) \geq c_\lambda n\Bigr] \geq 1-e^{-c_\lambda n}
  \]
  for every $n\geq 0$ and hence
  \[
\P_{w,\lambda,n}\Bigl[d(0,X_n) \geq c_\lambda' n\Bigr] \geq 1-e^{-c_\lambda' n}.
  \]
  for every $n\geq 0$.
\end{enumerate}
\end{thm}

\begin{remark}
The reader may find it an enlightening (and easy) exercise to prove the third item of \cref{thm:generalrepulsive} for simple random walk (that is, in the case $w(\omega)\equiv 1$) and observe what happens in the case $\lambda =1/2$.
\end{remark}
\subsection{About the proofs and organization}

The proof of \cref{thm:lambda,thm:generallambda} starts by using Fekete's Lemma to get bounds on \emph{bridges} at $z_{c,\lambda}$.  The fact that Fekete's lemma can be used to obtain surprisingly strong bounds for critical models on graphs of exponential growth was first exploited in  \cite{Hutchcroft2016944}, and is also central to our work on percolation in the nonunimodular setting \cite{Hutchcroftnonunimodularperc}. 
We then convert this control of bridges into a control of walks: this conversion centres around a tilted version of a generating function inequality of Madras and Slade \cite{MR2986656}. 
In \cref{sec:profit} we use \cref{thm:generallambda} to prove \cref{thm:generalrepulsive}. In \cref{sec:end} we examine the critically tilted case $\lambda=1/2$, posing several open problems and giving some small partial results. 

\label{subsec:abouttheproof}


\section{Proof of Theorems \ref{thm:lambda} and \ref{thm:generallambda}}
\label{sec:short}

Let $G$ be a connected, locally finite graph, let $\Gamma\subseteq \Aut(G)$ be transitive and nonunimodular, let $\root$ be a fixed root vertex, and let $w$ be a good weight function. We define the \textbf{height} of a vertex $v$ to be $\log \Delta(\root,v)$, and define
\[\bbH = \{ \log \Delta(u,v) : u,v\in V\}\]
to be the set of height differences that appear in $G$. We say that $v$ is \textbf{higher} than $u$ if $\Delta(u,v)>1$ and that $v$ is \textbf{lower} than $u$ if $\Delta(u,v) < 1$. We also define $\bbH_+=\bbH\cap (0,\infty)$ and $\bbH_{\geq0}$ to be the sets of positive and non-negative heights that appear in $G$ respectively.

We define the \textbf{level} $L_t$ to be the set of vertices at height $t$, that is,
\[
L_t = \{v\in V : \log \Delta(\root,v)=t\}.
\] We also define
\[t_0 = \sup \left\{ \log \Delta(u,v) : u\sim v \right\}\]
to be the maximum height difference between adjacent vertices,
 and define the \textbf{slab} $S_t$ to be
\[
S_t = \{v\in V : t \leq \log \Delta(\root,v) < t+t_0\}.
\]


We say that $\omega \in \Omega$ is an \textbf{up-bridge} if its endpoint $\omega^+$ is not lower than any of its other points, and its starting point $\omega^-$ is not higher any of its other points. Similarly, we say that $\omega \in \Omega$ is a \textbf{down-bridge} if its endpoint $\omega^+$ is not higher than any of its other points, and its starting point $\omega^-$ is not lower any of its other points. Thus, $\omega$ is an up-bridge if and only if its reversal $\omega^\leftarrow$ is a down-bridge.

We will specify that a walk $\omega$ is an up-bridge by writing a superscript of `u.b.', and a down-bridge by writing a superscript of `d.b.'.
For each $t\in \R$ we define
\begin{align*}
a_w(z;t) &= \sum_{\omega\in\Omega} w(\omega)z^{|\omega|}\mathbbm{1}\bigl( \omega:\root \xrightarrow{\textrm{u.b.}} L_t \bigr)
\intertext{and}
d_w(z;t) &= \sum_{\omega\in\Omega} w(\omega)z^{|\omega|}\mathbbm{1}\bigl( \omega:\root \xrightarrow{\textrm{d.b.}} L_{-t} \bigr),
\end{align*}
both of which are equal to zero when $t\notin\bbH_{\geq0}$. 
The $\Gamma$-invariance of $w$ and the tilted mass-transport principle implies that 
\begin{align*}
a_w(z;t)&= \sum_{v \in V} \mathbbm{1}\bigl( \Delta(0,v)=e^t \bigr) \sum_{\omega \in \Omega} w(\omega)z^{|\omega|} \mathbbm{1}\bigl( \omega:\root \xrightarrow{\textrm{u.b.}} v \bigr)\\
&=  \sum_{v \in V} \Delta(0,v) \mathbbm{1}\bigl( \Delta(v,0)=e^t \bigr) \sum_{\omega \in \Omega} w(\omega)z^{|\omega|} \mathbbm{1}\bigl( \omega:v \xrightarrow{\textrm{u.b.}} \root \bigr),
\end{align*}
and applying the reversibility of $w$ and the identity $\Delta(\root,v)=\Delta(v,\root)^{-1}$ we obtain that
\begin{align}
a_w(z;t)&=
 \sum_{v \in V} \Delta(0,v)\mathbbm{1}\bigl( \Delta(\root,v)=e^{-t} \bigr) \sum_{\omega \in \Omega} w(\omega^\leftarrow)z^{|\omega^\leftarrow|} \mathbbm{1}\bigl( \omega^\leftarrow : v \xrightarrow{\textrm{u.b.}} \root \bigr)
\nonumber\\
&= \sum_{v \in V} \Delta(0,v)\mathbbm{1}\bigl( \Delta(\root,v)=e^{-t} \bigr) \sum_{\omega \in \Omega} w(\omega)z^{|\omega|} \mathbbm{1}\bigl( \omega: \root \xrightarrow{\textrm{d.b.}} v \bigr)
\nonumber
\\
&=
e^{-t} d_w(z;t)\label{eq:atod}
\end{align}
 for every $t \in \bbH_{\geq0}$. Finally, define the sequence $(A_w(z;n))_{n\geq0}$ by
\[
A_w(z;n) = \sum_{\omega \in \Omega} w(\omega)z^{|\omega|} \mathbbm{1}\bigl[ \omega: 0 \xrightarrow{u.b.} S_{nt_0} \bigr] = \sum_{n t_0 \leq t < (n+1)t_0} a_w(z;t).
\]

\begin{lemma}
\label{lem:slicksupmult}
The sequence
$A_w(z;n)$ satisfies the generalized supermultiplicative estimate
\[
A_w(z;n+m+2) \geq z^2 A_w(z;n)A_w(z;m)
\]
for every $n,m\geq 0$ and $0\leq z\leq 1$.
\end{lemma}

\begin{proof}

It suffices to construct for each $n,m\geq 0$ an injective function
\begin{equation*}
\bigl\{ \omega_1 \in \Omega : \omega_1: \root \xrightarrow{\textrm{u.b.}} S_{nt_0} \bigr\} \times 
\bigl\{ \omega_2 \in \Omega : \omega_2: \root \xrightarrow{\textrm{u.b.}} S_{mt_0} \bigr\} \to \bigl\{ \omega \in \Omega : \omega: \root \xrightarrow{\textrm{u.b.}} S_{(n+m+2)t_0} \bigr\}
\end{equation*}
such that
\begin{equation}
\label{eq:phidesiredw}
w(\omega) \geq  w(\omega_1)w(\omega_2) 
\end{equation}
and
\begin{equation}
\label{eq:phidesiredlength}
 |\omega| \leq |\omega_1| + |\omega_2| +  2
\end{equation}
for every pair of up-bridges $\omega_1 : \root \xrightarrow{\textrm{u.b.}} S_{nt_0}$ and $\omega_2:\root \xrightarrow{\textrm{u.b.}} S_{mt_0}$. 

  For each $v\in V$,
we fix an automorphism $\gamma_v \in \Gamma$ such that $\gamma_v \root =v$. We let 
$\eta_1$ and $\eta_2$ be paths of length one and two respectively that start at $\root$ and whose endpoints have height $t_0$ and $2t_0$ respectively.
Let $\omega_1:\root \to S_{nt_0}$ and $\omega_2:\root \to S_{mt_0}$ be up-bridges. 
Then we have that 
\[  (n+m)t_0 \leq \log \Delta(\root,\omega^+_1)+\log \Delta(\root,\omega^+_2) < (n+m+2)t_0.\]

We define $\omega$ as follows:
\begin{itemize}
\item If the sum of the heights of $\omega_1^+$ and $\omega_2^+$ is greater than or equal to $(n+m+1)t_0$, we let $\omega$ be the composition $\omega_1 \circ \gamma_1 \eta_1 \circ \gamma_2 \omega_2$, where $\gamma_1=\gamma_{\omega_1^+}$ sends $\root$ to the final vertex of $\omega_1$, and $\gamma_2 = \gamma_{(\gamma_1 \eta_1)^+}$ sends $\root$ to the final vertex of $\gamma_1 \eta_1$. 
\item If the sum of the heights of $\omega_1^+$ and $\omega_2^+$ is strictly less than $(n+m+1)t_0$, we let $\omega$ be the composition $\omega_1 \circ \gamma_1 \eta_2 \circ \gamma_2 \omega_2$, where $\gamma_1=\gamma_{\omega_1^+}$ sends $\root$ to the final vertex of $\omega_1$, and $\gamma_2 = \gamma_{(\gamma_1 \eta_2)^+}$ sends $\root$ to the final vertex of $\gamma_1 \eta_2$. 
\end{itemize}
The path $\omega$ is clearly an up-bridge $0 \to S_{(n+m+2)t_0}$, and clearly satisfies the length bound \eqref{eq:phidesiredlength}, while the weight bound \eqref{eq:phidesiredw} follows from the assumption that $w$ is non-degenerate and zero-range. 
To see that the function $(\omega_1,\omega_2)\mapsto \omega$ is injective, observe that $\omega_1$ is necessarily equal to the longest initial segment of $\omega$ that has height strictly less than $(n+1)t_0$ at its endpoint, while $\omega_2$ is necessarily equal to the longest final segment of $\omega$ that has height difference strictly less than $(m+1)t_0$ between its starting point and endpoint.  \qedhere
\end{proof}

Next, we define $H^+_n$ to be the \textbf{upper half-space} $H^+_n := \bigcup_{s\geq nt_0} L_s$, and define
\[
b_w(z;n) = \sum_{\omega \in \Omega} w(\omega)z^{|\omega|} \mathbbm{1}\bigl[ \omega: 0 \to H^+_n\bigr].
\]
The following lemma complements \cref{lem:slicksubmult} and will be used to prove continuity of $\lambda \mapsto z_{c,\lambda}(w)$.

\begin{lemma}
\label{lem:slicksubmult}
The sequence 
$b_w(z;n)$ satisfies the generalized submultiplicative estimate
\[
b_w(z;n+m+1) \leq b_w(z;n)b_w(z;m)
\]
for every $n,m\geq 0$ and $z\geq 0$.
\end{lemma}

\begin{proof}
For each $v\in V$, let $\gamma_v\in \Gamma$ be such that $\gamma_v \root = v$, and for each $n\geq 0$ let 
\[H_n^+(v)=\gamma_v H_n^+ = \bigl\{u \in V : \log \Delta(v,u) \geq n t_0\bigr\}.\]
Let $\omega:0\to H^+_{n+m+1}$. Write $\omega=\omega_1\circ \omega_2$, where $\omega_1$ is the portion of $\omega$ up until it enters $H^+_n$ for the first time and $\omega_2$ is the remaining portion of $\omega$. Then $\omega^+_1 \in S_{nt_0}$ by definition of $t_0$, and it follows that $H^+_{n+m+1} \subseteq  H^+_{m}(\omega_1^+)$. 
Thus, using repulsivity of $w$ to bound $w(\omega)\leq w(\omega_1)w(\omega_2)$, we have that
\begin{align*}
b_w(z;n+m+1) \leq \sum_{\omega_1:\root \to H^+_n} w(\omega_1)z^{|\omega_1|}  \sum_{\omega_2:\omega_1^+ \to   H^+_m(\omega_1^+) } w(\omega_2)z^{|\omega_2|}  = b_w(z;n) b_w(z;m),
\end{align*}
where transitivity of $\Gamma$ and $\Gamma$-invariance of $w$ are used in the second line.
\end{proof}

We now recall Fekete's lemma, one form of which states that if $(c_n)_{n\geq 0}$ is a sequence taking values in $[-\infty,\infty]$ that  satisfies the generalized subadditive inequality
\[
c_{n+m+n_0} \leq c_n + c_m + C
\]
for some constants $n_0$ and $C$, then 
\[
\lim_{n\to\infty} \frac{1}{n}c_n = \inf_{n\geq0} \frac{c_n + C}{n+n_0} \in [-\infty,\infty],
\]
so that in particular the limit on the left-hand side exists. 

Applying Fekete's Lemma
  in light of \cref{lem:slicksupmult,lem:slicksubmult}, we obtain that the quantities
\begin{align}
\alpha_w(z)= - \lim_{n\to\infty} \frac{1}{n t_0} \log A_w(z;n) 
\quad \text{ and } \quad
\beta_w(z) =  - \lim_{n\to\infty} \frac{1}{n t_0} \log b_w(z;n) 
\end{align}
are well-defined (as elements of $[-\infty,\infty]$) and that
\begin{equation}
\label{eq:slickalpha}
A_w(z;n) \leq \frac{1}{z^2} e^{-\alpha_w(z)t_0(n+2)} \quad \text{ and } \quad b_w(z;n) \geq e^{-\beta_w(z)t_0(n+1)}
\end{equation}
for every $0\leq z \leq 1$ and $n\geq 0$.  Note that we trivially have that $A_w(z;n)\leq b_w(z;n)$ and hence that 
\begin{equation}
\label{eq:alphageqbeta}
\alpha_w(z)\geq \beta_w(z)
\end{equation} for every $0\leq z \leq 1$.

\begin{lemma}
\label{lem:slickbetabound}
Let $G$ be a connected, locally finite graph, let $\Gamma \subseteq \Aut(G)$ be transitive and nonunimodular, and let $w:\Omega \to [0,\infty)$ be a good weight function.
Then $\chi_w(z,\lambda)<\infty$ if and only if $\beta_w(z)>\max\{\lambda,1-\lambda\}$. In particular, $\beta_w(z_{c,\lambda}(w))\leq\max\{\lambda,1-\lambda\}$.
\end{lemma}

\begin{proof}
By definition,
\[
\chi_w(z,\lambda) = \sum_{v \in V}  \sum_{\omega \in \Omega} w(\omega) z^{|\omega|} \mathbbm{1}\bigl(\omega : \root \to v \bigr) \Delta^\lambda(0,v) 
\]
for every $\lambda \in \R$ and $z \geq 0$.  Using the assumptions that $\Gamma$ is transitive and $w$ is $\Gamma$-invariant and reversible, we apply the the tilted mass-transport principle to obtain that
\begin{multline*}
\sum_{v \in V}  \sum_{\omega \in \Omega} w(\omega) z^{|\omega|} \mathbbm{1}\bigl( \Delta(\root,v) < 1,\, \omega : \root \to v \bigr) \Delta^\lambda(0,v)\\ = \sum_{v \in V}  \sum_{\omega \in \Omega} w(\omega) z^{|\omega|} \mathbbm{1}\bigl( \Delta(\root,v) > 1,\, \omega : \root \to v \bigr) \Delta^{1-\lambda}(\root,v),
\end{multline*}
and hence that 
\begin{align*}
\chi_w(z,\lambda) &\asymp \sum_{v \in V}  \sum_{\omega \in \Omega} w(\omega) z^{|\omega|} \mathbbm{1}\bigl( \Delta(\root,v) \geq 1,\, \omega : \root \to v \bigr) \left[\Delta^\lambda(\root,v) +\Delta^{1-\lambda}(\root,v)\right]\\
&\asymp \sum_{v \in V}  \sum_{\omega \in \Omega} w(\omega) z^{|\omega|} \mathbbm{1}\bigl( \Delta(\root,v) \geq 1,\, \omega : \root \to v \bigr) \Delta^{\max\{\lambda,1-\lambda\}}(\root,v).
\end{align*}
Indeed, $\chi_w(z,\lambda)$ is upper bounded by twice the right-hand side of the second line and lower bounded by half the right-hand side of the second line. 
 We deduce that
\[
\chi_w(z,\lambda) \leq \sum_{n\geq 0} e^{\max\{\lambda,1-\lambda\}t_0n} b_w(z;n),
\]
and consequently that $\chi_w(z,\lambda)<\infty$ if $\beta_w(z) > \max\{\lambda,1-\lambda\}$. 
 Conversely, if $\chi_w(z,\lambda)<\infty$ then we have that
 \begin{equation*}
 \limsup_{n\to\infty}
e^{\max\{\lambda,1-\lambda\}t_0 n} b_w(z;n) \leq \limsup_{n\to\infty} 
 \sum_{v\in H_n^+} \sum_{\omega \in \Omega} w(\omega) z^{|\omega|} \mathbbm{1}\bigl(  \omega : \root \to v \bigr) \Delta^{\max\{\lambda,1-\lambda\}}(\root,v) =0,
 \end{equation*}
and hence by \eqref{eq:slickalpha} that $\beta_w(z) > \max\{\lambda,1-\lambda\}$ whenever $\chi_w(z,\lambda)<\infty$. 

Finally, the bound $\beta_w(z_{c,\lambda}(w))\leq \max\{\lambda,1-\lambda\}$ follows since  $\chi_w(z_{c,\lambda},\lambda)=\infty$ by submultiplicativity of $Z_w(\lambda,n)$.
\end{proof}

Moreover, we have the following.

\begin{lemma}
\label{lem:slickcontinuity}
Let $G$ be a connected, locally finite graph, let $\Gamma \subseteq \Aut(G)$ be transitive and nonunimodular, and let $w:\Omega \to [0,\infty)$ be a good weight function. Then  
$\alpha_w(z)$ is left-continuous on $(0,\infty)$ and $\beta_w(z)$ is right-continuous on $(0,z_t(w))$. Moreover, both $\alpha_w(z)$ and $\beta_w(z)$ are strictly decreasing when they are positive. 
\end{lemma}

\begin{proof}
Both $\alpha_w(z)$ and $\beta_w(z)$ are clearly decreasing in $z$ for $z>0$. 
For each $n\geq 0$, $A_w(z;n)$ and $b_w(z;n)$ are both defined as power series in $z$ with non-negative coefficients. It follows that they are each left-continuous in $z$ for $z>0$ and are continuous in $z$ within their respective radii of convergence, which are always at least $z_t(w)$ by the trivial bound
\[
A_w(z;n)\leq b_w(z;n)\leq \chi_w(z,\lambda),
\] 
which holds for every $z,n,\lambda \geq 0$.
If $z>0$ and we define \[\alpha_w(z-):= \inf_{0<\eps < z} \alpha_w(z-\eps) = \lim_{\eps\downarrow 0} \alpha_w(z-\eps),\]
then \eqref{eq:slickalpha} implies that
\[ A_w(z-\eps;n) \leq \frac{1}{(z-\eps)^2} e^{-\alpha_w(z-\eps)t_0(n+2)} \leq \frac{1}{(z-\eps)^2} e^{-\alpha_w(z-)t_0(n+2)} \]
for every $n\geq 0$.
It follows by left-continuity of $A_w(z;n)$ that the bound
\[
A_w(z;n) \leq \frac{1}{z^2} e^{-\alpha_w(z-)t_0(n+2)}
\]
also holds. This implies that $\alpha_w(z-) \leq \alpha_w(z)$ for every $z>0$, which is equivalent to left-continuity since $\alpha_w(z)$ is decreasing. The proof of the claim concerning right-continuity of $\beta_w(z)$ on $(0,z_t(w))$ is similar.

The claim that $\alpha_w(z)$ and $\beta_w(z)$ are strictly decreasing when they are positive follows from the trivial inequalities
\[A_w(z';n) \geq \left(\frac{z'}{z}\right)^{n} A_w(z;n) \quad \text{ and } \quad b_w(z';n) \geq \left(\frac{z'}{z}\right)^{n} b_w(z;n),\]
which hold for every $n\geq 0$ and $z'\geq z > 0$.
\end{proof}

\cref{lem:slickcontinuity} has the following very useful consequence.

\begin{lemma}
\label{cor:slickalphabound}
Let $G$ be a connected, locally finite graph, let $\Gamma \subseteq \Aut(G)$ be transitive and nonunimodular, and let $w:\Omega \to [0,\infty)$ be a good weight function. Then 
$\alpha_w(z_{c,\lambda}(w))\geq \max\{\lambda,1-\lambda\}$ for every $\lambda \in \R$.
\end{lemma}

\begin{proof} We trivially have that $\alpha_w(z)\geq \beta_w(z)$ for every $z >0$, and hence by \cref{lem:slickbetabound} that $\alpha_w(z) > \max\{\lambda,1-\lambda\}$ for every $0<z<z_{c,\lambda}(w)$. Thus, the claim follows by left-continuity of $\alpha_w(z)$. 
\end{proof}

\subsection{Relating walks and bridges}

We now wish to relate the quantities $\alpha_w(z)$ and $\beta_w(z)$. 
The following proposition, which is the central idea behind this proof,  is a tilted analogue of a well-known inequality relating generating functions for walks and bridges in $\Z^d$ due to Madras and Slade \cite[Corollary 3.1.8]{MR2986656} and related to the work of Hammersley and Welsh \cite{MR0139535}.  Indeed, the idea of the proof here is to combine that proof with a judicious use of the tilted mass-transport principle.

\begin{prop}
\label{prop:slick}
Let $G$ be a connected, locally finite graph, let $\Gamma \subseteq \Aut(G)$ be transitive and nonunimodular, and let $w$ be a good weight function. 
Then the estimate
\[
\chi_{w}(z,1/2) 
\leq \frac{1}{z}\exp\left[2\sum_{t \in \bbH_+} a_w(z;t)e^{t/2}  \right] \leq \frac{1}{z}\exp\left[2\sum_{n\geq0} A_w(z;n) e^{(n+1)t_0/2}\right]
\]
holds for every $z \geq 0$. 
\end{prop}

\begin{proof}
We say that $\omega$ is a \textbf{upper half-space walk} if $\Delta(\omega^-,\omega(i))>1$ for every $i>0$, that is, if $\omega$ is strictly higher than its starting point at every positive time. Similarly, we call $\omega$ a \textbf{reverse descent} if $\Delta(\omega^-,\omega(i))\geq 1$ for every $i>0$, that is, if $\omega$ is at least as high as its starting point at every positive time. A path $\omega$ is a \textbf{descent} if its reversal is a reverse descent. 
Define
\begin{align*}
h_w(z;t) &= \sum_{v\in V} \sum_{\omega \in \Omega} w(\omega)z^{|\omega|} \mathbbm{1}\bigl( \omega: \root \to L_t \text{ is an upper half-space walk}  \bigr),\\
r_w(z;t) &= \sum_{v\in V} \sum_{\omega \in \Omega} w(\omega)z^{|\omega|} \mathbbm{1}\bigl( \omega: \root \to L_t \text{ is a reverse descent}  \bigr),
\end{align*}
and
\[
\cH_w(z,\lambda) = \sum_{t\in \bbH_{\geq0}} h_w(z;t)e^{\lambda t}.
\]
Note that $h_w(z;0)=1$, as the only upper half-space walk ending in $L_0$ is the trivial path at $\root$.

Let $\eta$ be a path of length $1$ ending in $\root$ whose starting point $\eta^-=v$ has height $-t_0$. Then for every $t\geq 0$ and every reverse descent $\omega: \root \to L_t$, the composition $\eta \circ \omega$ is an upper half-space walk $\eta\circ \omega : v \to L_{t}$. Using the fact that $w$ is zero-range, non-degenerate, and $\Gamma$-invariant, this yields the inequality
\begin{equation}
\label{eq:rtoh}
h_w(z;t+t_0) \geq z r_w(z;t)
\end{equation}
for every $t\in \bbH_{\geq0}$.

Let $\omega \in \Omega$. Let $\omega_1$ be the portion of $\omega$ up until the last time that it visits a point of minimal height, and let $\omega_2$ be the remaining portion of $\omega$, so that $\omega=\omega_1\circ \omega_2$. This decomposition is defined in such a way that $\omega_1$ is a descent and $\omega_2$ is an upper half-space walk. 
Thus, using repulsivity, $\Gamma$-invariance, and reversibility of $w$, we have that
\begin{align*}
\chi_w(z,\lambda) &\leq \sum_{t\in \bbH_{\geq 0}} w(\omega_1) z^{|\omega_1|} \mathbbm{1}\bigl( \root \xrightarrow{d.} L_{-t} \bigr) \sum_{s\in \bbH_{\geq 0}} w(\omega_2)z^{|\omega_2|}\mathbbm{1}\bigl(\omega_1^+ \xrightarrow{\textrm{u.h.s.}} L_{s}(\omega_1^+)  \bigr) e^{\lambda(s-t)}\\
&=
\sum_{t\in \bbH_{\geq 0}} w(\omega_1) z^{|\omega_1|} \mathbbm{1}\bigl( \root \xrightarrow{d.} L_{-t} \bigr) e^{-\lambda t} \sum_{s\in \bbH_{\geq 0}} h_w(z;s) e^{\lambda s}\\
& = \sum_{t\in \bbH_{\geq 0}} r_w(z;t)e^{(1-\lambda)t} \sum_{s\in \bbH_{\geq 0}} h_w(z;s) e^{\lambda s},
\end{align*}
where the superscripts d.\ and u.h.s.\ denote descents and upper half-space walks respectively and where the tilted mass-transport principle is used in the final equality. Applying \eqref{eq:rtoh} we obtain that
\[
\chi_w(z,\lambda) \leq \frac{1}{z} \cH_w(z,\lambda)\cH_w(z,1-\lambda)
\]
for every $z>0$ and $\lambda \in \R$.
Thus, to conclude the proof of the present proposition, it suffices to prove that the inequality
\begin{equation}
\label{eq:Hdesired}
\cH_w(z,1/2) \leq \exp\left[\sum_{t \in \bbH_{\geq 0}}a_w(z;t)e^{t/2}\right]
\end{equation}
holds for every $z\geq 0$.

Let $t\in \bbH_+$ and let $\omega:\root \to L_t$ be an upper half-space walk.  We decompose $\omega=\omega_1 \circ \omega_2 \circ \cdots \circ \omega_k$ for some $k\geq 1$ recursively as follows: We first define $\omega_1$ to be the portion of $\omega_1$ up until the last time it attains its maximum height. Now suppose that $i\geq 2$. If $\omega_1 \circ \cdots \circ \omega_{i-1}=\omega$, we stop. Otherwise, consider the piece of $\omega$ that remains after $\omega_1 \circ \cdots \circ \omega_{i-1}$. If $i$ is \emph{odd}, let $\omega_i$ be the portion of this piece up to the last time it attains its \emph{maximum} height. If $i$ is \emph{even}, let $\omega_i$ be the portion of this piece up to the last time it attains its \emph{minimum} height.

Let $s(\omega_i)$ be the absolute value of the height difference between $\omega_i^-$ and $\omega_i^+$. 
Observe that for each $i\leq k$, $\omega_i$ is an up-bridge if $i$ is odd and a down-bridge if $i$ is even. Moreover, the sequence $s(\omega_i)$ is decreasing and satisfies $\sum_{i=1}^k (-1)^{i+1} s(\omega_i) =t$. 
This leads to the bound
\[
h_w(z;t) \leq \sum_{k \geq 1} \sum_{s \in S_{t,k}} \prod_{i=0}^{\lfloor (k-1)/2\rfloor} a_w(z;s_{2i+1}) \prod_{i=1}^{\lfloor k/2 \rfloor} d_w(z;s_{2i}),
\]
where we define $S_{t,k}$ to be the set of decreasing sequences  $s_1,\ldots,s_k$ in $\bbH_+$ such that $\sum_{i=1}^k (-1)^{i+1} s_i =t$. Now, observe that 
\[
a_w(z;t) = e^{-t/2} \sqrt{a_w(z;t) d_w(z;t)} \quad \text{ and } \quad d_w(z;t) = e^{t/2}\sqrt{a_w(z;t) d_w(z;t)},
\]
 by \eqref{eq:atod}, and hence  that
\begin{align*}
h_w(z;t) &\leq \sum_{k \geq 1} \sum_{s \in S_{t,k}} \prod_{i=1}^{k} e^{(-1)^{i} s_i/2 }\sqrt{a_w(z;s_i) d_w(z;s_i)}
\nonumber
  &
 = e^{-t/2}\sum_{k \geq 1} \sum_{s \in S_{t,k}} \prod_{i=1}^{k} \sqrt{a_w(z;s_i) d_w(z;s_i)}
\end{align*}
for every $t \in \bbH_+$.
Now, let $S_k$ be the set of all decreasing sequences $s=s_1,\ldots,s_k$ of elements of $\bbH_+$, and observe that for any non-negative function $f : \bbH_+ \to [0,\infty]$ we have that
\[
  \prod_{t\in \bbH_+} (1+f(t)) = 1+ \sum_{k\geq 1} \sum_{s\in S_k} \prod_{i=1}^k f(s_i)= 1+\sum_{t\in \bbH_+} \sum_{k\geq 1} \sum_{s\in S_{t,k}} \prod_{i=1}^k f(s_i).
\] 
Applying this equality with $f(t)=\sqrt{a_w(z;t) d_w(z;t)}$ we obtain that
\begin{align*}
\cH_w(z,1/2) &= 1+\sum_{t\in \bbH_+} h_w(z;t) e^{t/2} \leq
1+\sum_{t\in \bbH_+} \sum_{k \geq 1} \sum_{s \in S_{t,k}} \prod_{i=1}^{k} \sqrt{a_w(z;s_i) d_w(z;s_i)}
\\&= \prod_{t\in \bbH_+}\left(1 +  \sqrt{a_w(z;t)d_w(z;t)} \right) = \prod_{t\in \bbH_+}\left(1 + e^{t/2}a_w(z;t)  \right).
\end{align*}
Using the elementary inequality $1+x\leq e^x$ concludes the proof. 
\end{proof}

The following is an immediate consequence of \cref{cor:slickalphabound}, \cref{prop:slick}, and the estimate \eqref{eq:slickalpha}. 

\begin{corollary} 
\label{cor:slick}
If $z\geq 0$ is such that $\alpha=\alpha_w(z)>1/2$ then
\[
\chi_w(z,1/2) \leq \frac{1}{z}\exp \left[ \frac{2}{ z^2 e^{\alpha t_0} \bigl[e^{(\alpha-1/2)t_0}-1\bigr]} 
\right]< \infty.
\]
In particular, $z_{c,\lambda}(w)<z_{c,1/2}(w)=z_{t}(w)$ for every $\lambda \neq 1/2$.
\end{corollary}

We now apply \cref{cor:slick} to prove \cref{thm:lambda,thm:generallambda}.

\begin{proof}[Proof of \cref{thm:lambda,thm:generallambda}]
For each $v\in V$, let $\gamma_v\in \Gamma$ be an automorphism with $\gamma_v \root =v$, and for each $t\in \R$ and $v\in V$ let 
\[S_t(v)=\gamma_vS_t = \bigl\{u\in V : t\leq \Delta(v,u) < t+t_0\bigr\}.\]
Write $z_{c,\lambda}=z_{c,\lambda}(w)$.
 It follows from \cref{cor:slick} and \cref{cor:slickalphabound} that 
\[ \sum_{\omega \in \Omega} w(\omega)z_{c,\lambda}^{|\omega|} \mathbbm{1}\bigl[ \omega: 0 \to S_0 \cup S_{t_0} \bigr] \leq \chi_w(z_{c,\lambda},1/2)<\infty.
\]
Now suppose that $t >0$ and that $\omega:0\to S_t$ is a path. 
Then we can decompose $\omega=\omega_1\circ\omega_2\circ\omega_3$, where $\omega_1 : 0\to S_0$, the path $\omega_2: \omega_1^+ \to S_{t-t_0} \subseteq S_{t-t_0}(\omega_1^+) \cup S_{t-2t_0}(\omega_1^+)$ is an up-bridge, and $\omega_3 : \omega_2^+\to S_t \subseteq S_0(\omega_2^+) \cup S_{t_0}(\omega_2^+)$. Indeed, simply take $\omega_1$ to be the portion of $\omega_1$ up to the last visit to $S_0$, take  $\omega_2$ to be the portion of $\omega$ between the last visit to $S_0$ and the first subsequent visit to $S_{t-t_0}$, and take $\omega_3$ to be the remaining final piece (it is possible that some of these paths have length zero, but this is not a problem). By summing over possible choices of $\omega_1,\omega_2$ and $\omega_3$, and using both transitivity of $\Gamma$ and $\Gamma$-invariance and repulsivity of $w$, we obtain that
\begin{multline*} \sum_{\omega \in \Omega} w(\omega) z_{c,\lambda}^{|\omega|} \mathbbm{1}\bigl[\omega:0\to S_t\bigr]\\ \leq
\Bigl( \sum_{\omega \in \Omega} w(\omega)z_{c,\lambda}^{|\omega|} \mathbbm{1}\bigl[ \omega: 0 \to S_0\cup S_{t_0} \bigr] \Bigr)^2 \left[A_w(z;t-t_0) + A_w(z;t-2t_0)\right].
\end{multline*}
When $\lambda \neq 1/2$ the prefactor on the right-hand side is finite and does not depend on $t$, and we deduce easily that $\beta_w(z_{c,\lambda}) \geq \alpha_w(z_{c,\lambda})$, and hence that $\beta_w(z_{c,\lambda}) = \alpha_w(z_{c,\lambda})$ by \eqref{eq:alphageqbeta}. It then follows from \cref{lem:slickbetabound,cor:slickalphabound} that
\begin{equation}
\label{eq:alpha=beta}\alpha_w(z_{c,\lambda})=\beta_{w}(z_{c,\lambda})=\max\{\lambda,1-\lambda\}\end{equation}
for every $\lambda \neq 1/2$.  
Using left-continuity of $\alpha_w(z)$ and right-continuity of $\beta_w(z)$ from \cref{lem:slickcontinuity}, this implies that $\alpha_w(z)$ is a continuous, strictly decreasing function $(0,z_t] \mapsto [1/2,\infty)$ whose inverse is given by $\lambda \mapsto z_{c,\lambda}$. This implies that the latter function is continuous and strictly increasing on $(-\infty,1/2]$ as claimed. 
\end{proof}

\section{Critical exponents, two-point function decay and ballisticity}
\label{sec:profit}

\subsection{Counting walks}

Let $G$ be a connected, locally finite graph, let $\Gamma \subseteq \Aut(G)$ be transitive and nonunimodular, and let $w:\Omega \to [0,\infty)$ be a good weight function.
For each $z\geq 0$, the \textbf{bubble diagram} is defined to be the $\ell^2$-norm of the two-point function, that is,
\[
B_w(z) = \sum_{x \in V} G_w(x)^2.
\]
Convergence of the bubble diagram at $z_c$ is well-known to be a signifier of mean-field behaviour for the self-avoiding walk, see \cite[Section 1.5]{MR2986656}. 

The following lemma allows us to easily deduce the convergence of the bubble diagram at $z_c$ from \cref{thm:lambda,thm:generallambda}.

\begin{lemma}
\label{lem:tiltedbubble}
Let $G$ be a connected, locally finite graph, let $\Gamma$ be a transitive nonunimodular subgroup of $\Aut(G)$, and let $w:\Omega\to [0,\infty)$ be a good weight function. Then 
\[B_w(z) \leq \chi_w(z,\lambda)^2\]
for every $\lambda \in \R$. In particular, if $0\leq z < z_t(w)$ then $B_w(z)<\infty$. 
\end{lemma}

Note that it is always best to take $\lambda=1/2$ when applying this bound.

\begin{proof}
We can express
\[
B_w(z) = \sum_{x\in V}\sum_{\omega: \root \to x} w(\omega_1)z^{|\omega_1|} \sum_{\omega: x \to \root} w(\omega_1)z^{|\omega_1|}.
\]
Since $\Delta^\lambda(\root,\root)=1$ we have the trivial bound
\begin{align*}
B_w(z) &\leq  \sum_{x\in V}\sum_{\omega: \root \to x} w(\omega_1)z^{|\omega_1|} \sum_{ y\in V} \sum_{\omega: x \to y} w(\omega_1)z^{|\omega_1|} \Delta^\lambda(\root,y)\\
&= \sum_{x\in V}\sum_{\omega: \root \to x} w(\omega_1)z^{|\omega_1|} \Delta^\lambda(\root,x) \sum_{ y\in V} \sum_{\omega: x \to y} w(\omega_1)z^{|\omega_1|} \Delta^\lambda(x,y) = \chi_w(z,\lambda)^2,
\end{align*}
where the cocycle identity was used in the second equality and $\Gamma$-invariance of $w$ was used in the third. 
\end{proof}

The following differential inequality, which is classical for self-avoiding walk on $\Z^d$ \cite[Lemma 1.5.2]{MR2986656}, allows us to deduce \cref{thm:SAWcounting} and item (1) of both \cref{thm:tilted,thm:generalrepulsive} from \cref{thm:lambda,thm:generallambda}.

\begin{lemma} Let $G$ be a connected, locally finite graph, let $\Gamma \subseteq \Aut(G)$ be a transitive group of automorphisms, and let $w:\Omega\to[0,\infty)$ be a good weight function. Then for every $\lambda\in \R$ and $z\in [0,z_{c,\lambda}(w))$, we have that
\[\frac{\chi_w(z,\lambda)^2}{\mathsf{B}_w(z)} \leq \frac{\partial}{\partial z}\left[z \chi_w(z,\lambda) \right] \leq \chi_w(z,\lambda)^2.\]
\end{lemma}
The proof is closely adapted from the proof given in \cite[Lemma 1.5.2]{MR2986656}. We simply use the cocycle identity and the tilted mass-transport principle to `take the modular function along for the ride'.

\begin{proof}
For every $0\leq z< z_{c,\lambda}(w)$ we have that
\[
\frac{\partial}{\partial z} \left[z\chi_w(z,\lambda)\right] = \sum_{v\in V} \sum_{\omega :\root \to v} (|\omega|+1)w(\omega)z^{|\omega|} \Delta^\lambda(\root,v).
\]
Since $|\omega|+1$ is the number of ways to split $\omega$ into two (possibly length zero) subpaths, we deduce that
\[
\frac{\partial}{\partial z} \left[z\chi_w(z,\lambda)\right] = \sum_{u,v\in V} \sum_{\omega_1: \root \to u} \sum_{\omega_2 : u \to v} w(\omega_1\circ \omega_2)z^{|\omega_1|+|\omega_2|} \Delta^\lambda(\root,v).
\]
For the upper bound, we use repulsivity and the cocycle identity to write 
\begin{align*}
\frac{\partial}{\partial z} \left[z\chi_w(z,\lambda)\right] &\leq \sum_{u\in V} \sum_{\omega_1:\root\to u} w(\omega_1) z^{|\omega_1|} \Delta^\lambda(\root,u)
\sum_{u\in V} \sum_{\omega_2: u\to v} w(\omega_2) z^{|\omega_2|} \Delta^\lambda(u,v) = \chi_w(z,\lambda)^2,
\end{align*}
where the equality on the second line follows by transitivity of $\Gamma$ and $\Gamma$-invariance of $w$. 

We now turn to the lower bound. 
We begin by applying the tilted mass-transport principle to the sum over $u$ to deduce that
\begin{align*}
\frac{\partial}{\partial z} \left[z\chi_w(z,\lambda)\right] &= \sum_{u,v\in V} \sum_{\omega_1:\root \to u}\sum_{\omega_2:\root\to v} w(\omega_1\circ \omega_2)z^{|\omega_1|+|\omega_2|} \Delta^\lambda(u,v)  \Delta(\root,u)\\
&= \sum_{u,v\in V} \sum_{\omega_1:u \to \root }\sum_{\omega_2:\root\to v} w(\omega_1\circ \omega_2)z^{|\omega_1|+|\omega_2|}  \Delta^{1-\lambda}(\root,u) \Delta^\lambda(\root,v),
\end{align*}
where the cocycle identity has been used in the second line. 
Since $w$ is zero-range, we can bound
\begin{equation*}
w(\omega_1 \circ \omega_2) \geq w(\omega_1) w(\omega_2) \mathbbm{1}\bigl( \omega_1,\omega_2 \text{ disjoint}\bigr)=w(\omega_1)w(\omega_2)\left[1-\mathbbm{1}\bigl(  \omega_1,\omega_2 \text{ not disjoint}\bigr)\right].
\end{equation*}
If $\omega_1:u\to\root$ and $\omega_2:\root \to v$ are \emph{not} disjoint, then there exists $x \in V$ and paths $\omega_{1,1}:u \to x$, $\omega_{1,2}: x \to \root$, $\omega_{2,1}:\root \to x$ and $\omega_{2,2}:x\to v$ such that $\omega_1=\omega_{1,1}\circ\omega_{1,2}$, $\omega_2=\omega_{2,1}\circ\omega_{2,2}$, such that $\omega_{1,1}$ and $\omega_{2,2}$ are disjoint, and such that neither $\omega_{1,2}$ or $\omega_{2,1}$ is trivial. Indeed, simply take $\omega_{1,1}$ to be the portion of $\omega_1$ up until the first time it intersects $\omega_2$, and let $\omega_{2,1}$ be the portion of $\omega_2$ up until the last time it visits $\omega_{1,2}^+$. It follows that
\begin{multline*}
\sum_{u,v\in V} \sum_{\omega_1: u \to \root} \sum_{\omega_2:\root \to v}  \mathbbm{1}\bigl[\omega_1,\omega_2 \text{ not disjoint}\bigr] w(\omega_1)w(\omega_2)z^{|\omega_1|+|\omega_2|} \Delta^{1-\lambda}(\root,u)  \\
\hspace{2cm}\leq \sum_{x,u,v\in V} \sum_{\substack{\omega_{1,1}: u \to x\\ \omega_{1,2}: x \to \root\\
|\omega_{1,2}|\geq1}} \sum_{\substack{\omega_{2,1}: \root \to x\\ \omega_{2,2}: w \to v\\
|\omega_{2,1}|\geq1}} \mathbbm{1}\bigl[\omega_{1,1},\omega_{2,1} \text{ disjoint}\bigr]
 w(\omega_{1,1}\circ \omega_{1,1})w(\omega_{2,2} \circ \omega_{2,2})
\\ 
 \hspace{1cm}\cdot z^{|\omega_{1,1}|+|\omega_{1,2}|+|\omega_{2,1}|+|\omega_{2,2}|} \Delta^{1-\lambda}(\root,u) \Delta^\lambda(\root,v).
\end{multline*}
Using repulsivity of $w$, the zero-range property, and the cocycle identity, we can bound
\begin{multline*}
\mathbbm{1}\bigl[\omega_{1,1},\omega_{2,2} \text{ disjoint}\bigr]
 w(\omega_{1,1} \circ \omega_{1,2})w(\omega_{2,1} \circ \omega_{2,2}) \Delta^{1-\lambda}(\root,u) \Delta^\lambda(\root,v) \\\leq w(\omega_{1,1}\circ \omega_{2,2}) w(\omega_{1,2})w(\omega_{2,1}) \Delta(\root,x)\Delta^{1-\lambda}(x,u)\Delta^\lambda(x,v).
\end{multline*}
This leads to the bound
\begin{align*}
&\sum_{u,v\in V} \sum_{\omega_1: u \to \root} \sum_{\omega_2:\root \to v} \mathbbm{1}\bigl[\omega_1,\omega_2 \text{ not disjoint}\bigr] w(\omega_1)w(\omega_2)z^{|\omega_1|+|\omega_2|} \Delta^{1-\lambda}(\root,u) \\
&\leq \sum_{x\in V} \left[G_w(x) \!-\! \mathbbm{1}\left(x=\root\right)\right]^2 \Delta(\root,x) \sum_{u,v \in V} w(\omega_{1,1} \circ \omega_{2,2}) z^{|\omega_{1,1}|+|\omega_{2,2}|} \Delta^{1-\lambda}(0,u)\Delta^\lambda(\root,v)\\
&= \frac{\partial}{\partial z}\left[ z \chi_w(z,\lambda)\right] \sum_{x\in V} \left[G_w(x)^2 - \mathbbm{1}\left(x=\root\right)\right]  \Delta(\root,x)
 = (B_w(z)-1) \frac{\partial}{\partial z}\left[ z \chi_w(z,\lambda)\right],
\end{align*}
where the tilted mass-transport principle is used in the final equality. (The $-\mathbbm{1}(x=\root)$ arises from the restriction that neither $\omega_{1,2}$ or $\omega_{2,1}$ is trivial.)
On the other hand, similar manipulations to those used in the upper bound, above, yield that
\begin{equation*}
\sum_{u,v\in V} \sum_{\omega_1: \root \to u} \sum_{\omega_2:\root \to v} w(\omega_1)w(\omega_2)z^{|\omega_1|+|\omega_2|} \Delta^{1-\lambda}(\root,u) \Delta^\lambda(\root,v) = \chi_w(z;\lambda)\chi_w(z;1-\lambda) = \chi_w(z;\lambda)^2.
\end{equation*}
Combining these inequalities, we deduce that
\[
\frac{\partial}{\partial z} \left[ z \chi_w(z,\lambda) \right] \geq \chi_w(z,\lambda)^2 - (B_w(z)-1)\frac{\partial}{\partial z} \left[ z \chi_w(z,\lambda) \right],
\]
which rearranges to give the desired inequality.
\end{proof}

Integrating this differential inequality yields the following estimates; see \cite[Theorem 1.5.3]{MR2986656}.

\begin{corollary}
\label{cor:bubblebound}
Let $G$ be a connected, locally finite graph, let $\Gamma \subseteq \Aut(G)$ be a transitive group of automorphisms, and let $w:\Omega\to[0,\infty)$ be a good weight function. Then 
\[
\frac{z_{c,\lambda}}{z_{c,\lambda}-z} \leq \chi_w(z;\lambda) \leq B_w(z_{c,\lambda})\frac{z_{c,\lambda}}{z_{c,\lambda}-z} + B_w(z_{c,\lambda})
\]
for every $\lambda \in \R$ and $0\leq z < z_{c,\lambda}=z_{c,\lambda}(w)$.
\end{corollary}

We now have everything we need to deduce the upper bound on the susceptibility for \cref{thm:SAWcounting,thm:tilted,thm:generalrepulsive}. To use the susceptibility bounds to deduce the bounds on $Z_w(\lambda;n)$, we use the following lemma. It replaces the Tauberian theory that is typically used in the literature.

\begin{lemma}
\label{lem:chitoZ}
Let $(c_n)_{n\geq 0}$ be a non-negative submultiplicative sequence with generating function $\Phi(x)=\sum_{n\geq 0} x^n c_n$. Then
\[
 x^n c_n \leq \left[\frac{\Phi(y)}{n+1}\right]^2\left(\frac{x}{y}\right)^{2n}
\]
for every $m\geq 1$ and $x\geq y > 0$. 
\end{lemma}

When we apply this lemma we will take $y=nx/(n+1)$ so that $(x/y)^{2n}\approx e^2$ is of constant order.

\begin{proof}
We clearly have that
\[
\sum_{k=0}^n \frac{c_{n-k}x^{n-k} + c_kx^k}{2} = \sum_{k=0}^n c_kx^k \leq \left(\frac{x}{y}\right)^n \Phi(y),
\]
and so there must exist $0\leq k \leq n$ such that 
\[\frac{c_{n-k}x^{n-k} + c_k x^k}{2} \leq \left(\frac{x}{y}\right)^n \frac{\Phi(y)}{n+1}.\]
For this $k$ we have that
\[
c_n x^n \leq c_{n-k}x^{n-k}c_k x^k \leq \left[\frac{c_{n-k}x^{n-k} + c_k x^k}{2} \right]^2\leq \left[\frac{\Phi(y)}{n+1}\right]^2\left(\frac{x}{y}\right)^{2n}
\]
by submultiplicativity and the inequality of arithmetic and geometric means.
\end{proof}

\begin{proof}[Proof of \cref{thm:SAWcounting} and part $(i)$ of \cref{thm:tilted,thm:generalrepulsive}] \hspace{0.1em}\newline
 \cref{thm:generallambda} implies that the tilted susceptibility $\chi_w(z_{c,\lambda},1/2)$ is finite for every $\lambda \neq 1/2$, and it follows from \cref{lem:tiltedbubble}, and  \cref{cor:bubblebound} that
\begin{align}
\chi_w(z;\lambda) &\leq  \chi_{w}(z_{c,\lambda};1/2)^2\frac{z_{c,\lambda}}{z_{c,\lambda}-z} + \chi_{w}(z_{c,\lambda};1/2)^2.
\label{eq:finalChi}
\intertext{for every $0\leq z < z_{c,\lambda}=z_{c,\lambda}(w)$. Thus, applying \cref{lem:chitoZ} with $c_n=Z_w(\lambda;n)$, $x=z_{c,\lambda}(w)$, and $y=nz_{c,\lambda}(w)/(n+1)$ yields that}
Z_w(\lambda;n)&\leq \left[e^2\chi_w(z_{c,\lambda},1/2)^4+o(1)\right]\mu_{c,\lambda}^n 
\label{eq:finalZ}
\end{align}
as $n\to\infty$.
\end{proof}

\begin{remark}
\label{remark:quant}
In the untilted case $\lambda=0$, \cref{cor:slick} together with \eqref{eq:finalChi} and \eqref{eq:finalZ} yield that
\begin{align}
\chi(z) &\leq \mu_c^2\exp \left[ \frac{4\mu_c^2}{e^{3t_0/2} -e^{t_0}} \right] \frac{z_c}{z_c-z} + O(1) && z \nearrow z_c\\
Z(n) &\leq  
\mu_c^4\exp \left[ \frac{8\mu_c^2}{e^{3t_0/2} -e^{t_0}} +2 \right] \mu_c^{n} + o(\mu_c^n) && n \nearrow \infty.
\end{align}
In particular, these estimates hold with $t_0 \geq \log (k-1)$ for the product of the $k$-regular tree $T_k$ with an arbitrary transitive graph $G$. We have not attempted to optimize these constants.
\end{remark}

\subsection{Ballisticity and two-point function decay}

\begin{proof}[Proof of items 2 and 3 of \cref{thm:tilted,thm:generalrepulsive}]
 Observe that the trivial inequality 
\[
G_w(z;x) \leq \left(\frac{z}{z'}\right)^{d(\root,x)} G^k_w(z';x)
\]
holds for every $z\geq 0$ and $x\in V$. On the other hand, for $z<z_t$ we have that
\[
 G_w(z;x) \leq \chi_w(z;1/2) \Delta^{-1/2}(\root,x)
\]
for every $x\in V$, and it follows by symmetry that
\[
G_w(z,x) \leq \chi_w(z;1/2) \left[ \min\bigl\{\Delta(\root,x),\Delta(x,\root)\bigr\}\right]^{1/2} \leq \chi_w(z;1/2)
\]
for every $0<z<z_t$ and $x\in V$. Thus, we have that
\[
G_w(z,x)  \leq \chi_w(z';1/2) \left(\frac{z}{z'}\right)^{d(\root,x)}
\]
for every  $x\in V$ and $0<z<z'<z_t$, which implies the claim of item 2.

We now prove item 3.  We prove the claim in the case $\lambda>1/2$, the case $\lambda<1/2$ being similar. Let $\lambda>\lambda'>1/2$. Then we have that 
\begin{align*}
\P_{w,\lambda,n}\Bigl( \log \Delta(\root,X_n) \leq cn \Bigr) &= Z_w(\lambda;n)^{-1} \sum_{x\in H^-_{cn}} \sum_{\omega \in \Omega} w(\omega) \Delta^\lambda(\root,x) \mathbbm{1}\bigl[\omega:0\to H^-_{cn},\, |\omega|=n \bigr]\\
&\leq  Z_w(\lambda,n)^{-1}Z_w(\lambda';n) e^{(\lambda-\lambda')cn}.
\end{align*}
We deduce from item 1 of \cref{thm:generalrepulsive} that
\[
\P_{w,\lambda,n}\left( \log \Delta(\root,X_n) \leq cn \right) \leq C_{\lambda'}\left(
\frac{z_{c,\lambda}}{z_{c,\lambda'}}\right)^n  e^{(\lambda-\lambda')cn}.
\]
The result follows by fixing $\lambda>\lambda'>1/2$ and letting $c=c_{\lambda,\lambda'}$ be sufficiently small that $z_{c,\lambda}/z_{c,\lambda'} < e^{-c(\lambda-\lambda')}$. \qedhere
\end{proof}

\section{Remarks and open problems}
\label{sec:end}

Several interesting questions remain open concerning the behaviour of the \emph{critically tilted SAW}, that is, the case $\lambda =1/2$.

\subsection{Exponents at $\lambda=1/2$ are graph dependent}
In this subsection we briefly outline an example that shows that the exponent governing the critically tilted susceptibility depends on the choice of $G$ and $\Gamma$, and in particular that \cref{thm:tilted} cannot always be generalised to $\lambda=1/2$. This follows from a related analysis for percolation on trees with respect to two different choices of nonunimodular automorphism group that we performed in \cite{Hutchcroftnonunimodularperc}. (Note that on trees the tilted susceptibilities for SAW and percolation are equal for $0\leq z=p \leq 1$.)

Let $T$ be the $k$-regular tree. 
The most obvious choice of a nonunimodular transitive subgroup of $T$ is the group $\Gamma_\xi$ consisting of those automorphisms of $T$ that fix some given end $\xi$ of $T$. For the pair $(T,\Gamma_\xi)$ we can easily compute
\[\alpha (z) = -\log_{k-1}(z) \qquad \text{ and } \qquad z_{c,\lambda} = (k-1)^{-\max\{\lambda,1-\lambda\}}. \]
Moreover, for $z<z_{c,\lambda}$ we can compute the tilted susceptibility to be
\[
\chi(z,\lambda) = \frac{1-z^2}{(1-(k-1)^{1-\lambda}z)(1-(k-1)^\lambda z)}.
\]
Thus, we see that for $\lambda\neq 1/2$,
$\chi(z_{c,\lambda}-\eps,\lambda)$ grows like $\eps^{-1}$ as $\eps\to0$, as stated in \cref{thm:tilted}, while at $\lambda=1/2$ the denominator has a double root and we have instead that
%
\[\chi_{z_{c,1/2}-\eps,1/2}=\frac{k-2}{k-1} \eps^{-2}.\]
This shows that \cref{thm:tilted} cannot be extended in general to the case $\lambda=1/2$. 

We now describe a different transitive nonunimodular group of automorphisms on the four-regular tree. We define a $(1,1,2)$-\textbf{orientation} of $T$ to be a (partial) orientation of the edge set of $T$ such that every vertex has one oriented edge emanating from it, two oriented edges pointing into it, and one unoriented edge incident to it. Fix one such orientation of $T$, and let $\Gamma'$ be the group of automorphisms of $T$ that preserve the orientation. 
In \cite{Hutchcroftnonunimodularperc}, we compute that
\begin{align*}
z_{c,\lambda} &= \frac{2^\lambda+2^{1-\lambda}+1 - \sqrt{(2^\lambda+2^{1-\lambda}+1)^2-12}}{6} & \lambda \in \R \\
\intertext{and}
\chi(z,\lambda) &= \frac{1-3z^2}{1 - (2^\lambda+2^{1-\lambda}+1)z+3z^2} & \lambda \in \R,\, 0 \leq z<z_{c,\lambda}. \label{eq:orientedtreechi}
\end{align*}The denominator of this expression never has a double root, so that, in contrast to the previous example,
\begin{align*}
 \chi(z_{c,\lambda}-\eps,\lambda) &\asymp \eps^{-1} & \lambda \in \R,\, \eps \downarrow 0
\end{align*}

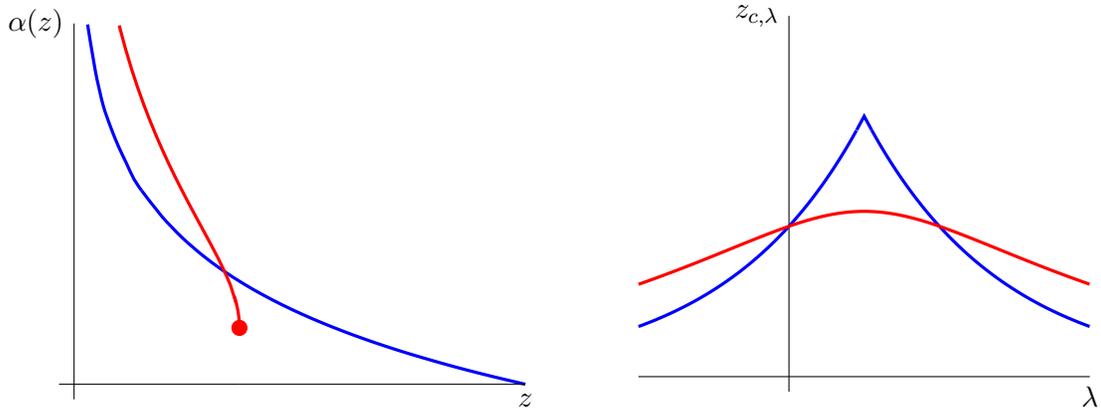
\begin{figure}
\centering
\begin{tikzpicture}
      \draw[-,thin] (-0.2,0) -- (6*0.366406859,0);
      \draw[-,thin] (0,-0.2) -- (0,4.8) node[left,black] {$\alpha(z)$};
      \draw[scale=6,domain=0.03:0.15,smooth,very thick,variable=\x,blue,samples=8] plot ({\x},{-0.25*ln(\x)/ln(3)});
      \draw[scale=6,domain=0.15:1,smooth,very thick,variable=\x,blue,samples=20] plot ({\x},{-0.25*ln(\x)/ln(3)});
      \draw[scale=6,domain=0.1:0.36,very thick,variable=\x,red,samples=20] plot ({\x},{0.25*log2((3*pow(\x,2)-\x+1+pow(9*pow(\x,4)-6*pow(\x,3)-pow(\x,2)-2*\x+1,0.5))/(2*\x))});
      \draw[scale=6,domain=0.355:0.3664,very thick,variable=\x,red,samples=30] plot ({\x},{0.25*log2((3*pow(\x,2)-\x+1+pow(9*pow(\x,4)-6*pow(\x,3)-pow(\x,2)-2*\x+1,0.5))/(2*\x))});
      \draw[scale=6,domain=0.366:0.3665,very thick,variable=\x,red,samples=30] plot ({\x},{0.25*log2((3*pow(\x,2)-\x+1+pow(max(9*pow(\x,4)-6*pow(\x,3)-pow(\x,2)-2*\x+1,0),0.5))/(2*\x))});
      \draw[-,black] (6*0.366406859,0) -- (6,0) node[below,black] {$z$};
      \draw[red,fill=red] (6*0.366406859,0.5*6*0.25) circle [radius=0.1]; 
\end{tikzpicture}
    \hspace{1cm}
\begin{tikzpicture}
      \draw[-,thin] (-2,0) -- (4,0) node[below] {$\lambda$};
      \draw[-,thin] (0,-0.2) -- (0,4.8) node[left,black] {$z_{c,\lambda}$};
      \draw[scale=2,domain=-1:0.45,smooth,very thick,variable=\x,blue,samples=100] plot ({\x},{3*pow(3,-max(\x,1-\x))});
      \draw[scale=2,domain=0.45:0.55,very thick,variable=\x,blue,samples=3] plot ({\x},{3*pow(3,-max(\x,1-\x))});
      \draw[scale=2,domain=0.55:2,smooth,very thick,variable=\x,blue,samples=100] plot ({\x},{3*pow(3,-max(\x,1-\x))});
      \draw[scale=2,domain=-1:2,smooth,very thick,variable=\x,red,samples=100] plot ({\x},{0.5*(pow(2,\x)+pow(2,1-\x)+1-pow(max(0,pow(pow(2,\x)+pow(2,1-\x)+1,2)-12),0.5))});
    \end{tikzpicture}
    \caption{
     Comparison of $\alpha(z)$ and $z_{c,\lambda}$ for the $4$-regular tree with respect to the automorphism group fixing an end (blue) and the automorphism group fixing a $(1,1,2)$-orientation (red). The second figure is formed by reflecting the first around the line $\alpha(z)=1/2$ and then rotating. 
    The intersection of the two curves on the left occurs at $(z_c,1)$. This intersection must occur since $z_{c,0}=z_{c,1}=z_c$ does not depend on the choice of automorphism group.}
\end{figure}


\noindent
 for every $\lambda \in \R$. 
Since $z_t$ for this example is smaller than $z_t$ for the previous example, we have by \cref{lem:tiltedbubble} that the bubble diagram converges at $z_t$, so that we could also have deduced this behaviour from \cref{cor:bubblebound}.
 Furthermore, it follows from our analysis of percolation in \cite{Hutchcroftnonunimodularperc} that $\alpha(z)$ has a jump discontinuity from $1/2$ to $-\infty$ at $z_t$. Indeed,
\begin{align*}\alpha(z)= &\begin{cases} \log_2\left(\frac{3z^2-z+1+\sqrt{9z^4-6z^3-z^2-2z+1}}{2z}\right) & z \leq z_t\\
-\infty & z>z_t.
\end{cases}  \end{align*}
As a further point of contrast to the previous example, we note that in this example the function $\lambda \mapsto z_{c,\lambda}$ is real-analytic, whereas in the previous example it was not differentiable at $\lambda=1/2$. 

\subsection{Hammersley-Welsh-type bounds for critically tilted SAW}
Consider SAW on $\Z^d$.
The Hammersley-Welsh inequality \cite{MR0139535} states that
\[
Z(n) \leq \exp\left[O(n^{1/2})\right]\mu_c^n.
\]
See \cite[Section 3.1]{MR2986656} for background and \cite{1708.09460} for a small improvement.

We now briefly outline how an analogous inequality may be obtained for critically tilted ($\lambda=1/2$) self-avoiding walk in the nonunimodular context.
It can be deduced from \cref{cor:slick} that
\[
\chi(z_t-\eps,1/2) \leq \exp\left[ O(\eps^{-1})\right],
\]
and applying \cref{lem:chitoZ} with $c_m=Z(1/2;n)$, $x=z_t$, and $y=z_t(1-m^{-1/2})$ yields that
\[
Z(1/2;n) \leq \exp\left[ O(n^{1/2})\right] \mu_{c,1/2}^n,
\]
which is an exact analogue of the Hammersley-Welsh bound. 

\subsection{Questions}

\begin{question}
Let $T_k$ be a $k$-regular tree, let $d\geq 1$ and consider the group of automorphisms $\Gamma_\xi \times \Aut(\Z^d) \subseteq \Aut(T_k \times \Z^d)$ of $T_k \times \Z^d$, where $\Gamma_\xi$ is the group of automorphisms that fix some specified end $\xi$ of $T$. 
\begin{enumerate}
	\item Is $B(z_t)<\infty$?
	\item What are the asymptotics of $a(z_t;n)$ and $b(z_t;n)$ as defined in \cref{sec:short}? 
	\item
What is the behaviour of $\chi(z_{t}-\eps,1/2)$ as $\eps\to0$? What about $Z(n,1/2)$ as $n\to\infty$?
\item What is the typical displacement of a SAW sampled from $\P_{1/2,n}$?
\item For which of these questions does the answer depend on $d$? 
\end{enumerate}
\end{question}

\begin{question}
Let $G$ be a connected locally finite graph and let $\Gamma \subseteq \Aut(G)$ be transitive and nonunimodular. Does there exist $C=C(G,\Gamma)< \infty$ such that
\[
Z(1/2;n) =  O\bigl(n^C \mu_c^n\bigr)
\]
for every $n\geq 1$? Is there a universal choice of this $C$? Does $C=1$ always suffice?
\end{question}

The question concerning $C=1$ arises from the guess that the pair $(T_k,\Gamma_\xi)$ has the largest subexponential correction to $Z(1/2;n)$ among all pairs $(G,\Gamma)$. 

\begin{question}
Let $G$ be a connected locally finite graph and let $\Gamma \subseteq \Aut(G)$ be transitive and nonunimodular. What asymptotics are possible for the typical displacement of a sample from $\P_{n,1/2}$? Is it always of order at least $n^{1/2}$? 
\end{question}

\subsection*{Acknowledgments} This work was carried out while the author was an intern at Microsoft Research, Redmond. 
 We thank Omer Angel for improving \cref{lem:chitoZ} by a factor of 4.  We also thank Tyler Helmuth for helpful discussions, and thank Gordon Slade and Hugo Duminil-Copin for comments on an earlier draft, and thank the anonymous referee for several helpful suggestions. 
  
  \setstretch{1}
  \bibliographystyle{abbrv}
{\footnotesize{
  \bibliography{unimodularthesis}
  }
}
\end{document}